\documentclass[12pt]{article}
\usepackage{amsmath,amsfonts,amssymb}
\usepackage{amsthm}
\usepackage{hyperref}

\makeatletter
 
  \@addtoreset{equation}{section}
 \makeatother

\newtheorem{thm}{Theorem}[section]
\newtheorem{lem}[thm]{Lemma}

\newtheorem{cor}[thm]{Corollary}

\theoremstyle{definition}

\let\abs=\envert
\newcommand{\floor}[1]{\left\lfloor#1\right\rfloor}

\theoremstyle{remark}

\begin{document}
\title{On equations $\sigma(n)=\sigma(n+k)$ and $\varphi(n)=\varphi(n+k)$
\footnote{MSC subject classification: 11A05, 11A25.}
\footnote{Key words: Arithmetic functions.}}
\author{Tomohiro Yamada}
\date{}
\maketitle

\begin{abstract}
We study the distribution of solutions of equations $\sigma(n)=\sigma(n+k)$ and $\varphi(n)=\varphi(n+k)$.
We give new upper bounds for these solutions.
\end{abstract}

\section{Introduction}
In this paper, we study equations $\sigma(n)=\sigma(n+k)$ and $\varphi(n)=\varphi(n+k)$.
As far as the author knows, an equation of these types was first referred by Ratat \cite{Rat},
who asked for which values of $n$ the equation $\varphi(n)=\varphi(n+1)$ holds and
gave $n=1, 3, 15, 104$ for examples.  In 1918, answering to Ratat's question, Goormaghtigh \cite{Goo}
gave $n=164, 194, 255, 495$.

After then, several authors such as Klee \cite{Kle}, Moser \cite{Mos}, Lal and Gillard \cite{LG}, Ballew, Case and Higgins \cite{BCH},
Baillie \cite{Bai1} \cite{Bai2} and Graham, Holt and Pomerance \cite{GHP} searched for solutions
to $\varphi(n)=\varphi(n+k)$.

Klee \cite{Kle} and Moser \cite{Mos} noted that if $p, 2p-1$ are both odd primes and $n=2(2p-1)$, then
$\varphi(n)=2p-2=\varphi(4p)=\varphi(n+2)$.  Under the quantitative prime $k$-tuplet conjecture,
the number of such primes $\leq x$ is $\gg x/(\log x)^2$.  Similarly to their result, we can see that
$\sigma(n)=\sigma(n+22)$ if $3l-1, 14l-1$ are both primes and $n=28(3l-1)$.

On the other hand, Erd\H{o}s, Pomerance and S\'{a}rk\"{o}zy \cite{EPS} showed that the number of solutions
$n\leq x$ to $\varphi(n)=\varphi(n+1)$ is at most $x\exp(-(\log x)^{1/3})$ for sufficiently large $x$
and a similar result holds for $\sigma(n)=\sigma(n+1)$.  They also conjectured that the number
of such solutions below sufficiently large $x$ is at least $x^{1-\epsilon}$ for every $\epsilon>0$.

Graham, Holt and Pomerance \cite{GHP} generalized these results.
They showed that, if $j$ and $j+k$ have the same prime factors with $g=\gcd (j, j+k)$, both of $jr/g+1, (j+k)r/g+1$
are primes which do not divide $j$ and
\begin{equation}\label{eq11}
n=j\left(\frac{j+k}{g}r+1\right),
\end{equation}
then $\varphi(n)=\varphi(n+k)$.
Moreover, they gave the corresponding result of \cite{EPS}.
According to them, we denote by $N(k, x)$ the number of integers $n\leq x$ with $\varphi(n)=\varphi(n+k)$
and $N_1(k, x)$ the number of integers $n\leq x$ with $\varphi(n)=\varphi(n+k)$ which are in the form
(\ref{eq11}).
Then they showed that $N_1(k, x)\leq x\exp(-(\log x)^{1/3})$ for sufficiently large $x$ and
$N(k, x)=N_1(k, x)$ if $k$ is odd, so that their result implies the result of \cite{EPS}.

Our purpose is to prove corresponding results on the equation $\sigma(a_1n+b_1)=\sigma(a_2n+b_2)$
and $\varphi(a_1n+b_1)=\varphi(a_2n+b_2)$ and improve the upper bound $x\exp(-(\log x)^{1/3})$
of \cite{GHP}.  The main results are the following two theorems.

\begin{thm}\label{th11}
Let $a_1, b_1, a_2, b_2$ be integers such that $a_1a_2(a_1b_2-a_2b_1)\neq 0$.

Assume that $m_1, m_2, k_1, k_2$ satisfy the relations
\begin{equation}\label{eq12}
m_1=\frac{k_2(a_1b_2-a_2b_1)}{a_2(k_2-k_1)}, m_2=\frac{k_1(a_1b_2-a_2b_1)}{a_1(k_2-k_1)},\\
k_1\sigma(m_1)=k_2\sigma(m_2).
\end{equation}
and $q_i=k_il-1 (i=1, 2)$ are both prime.
If
\begin{equation}\label{eq122}
n=\frac{m_1q_1-b_1}{a_1},
\end{equation}
then we have $a_1n+b_1=m_1q_1, a_2n+b_2=m_2q_2$ and $\sigma(a_1n+b_1)=\sigma(a_2n+b_2)$.

Similarly, assume that $m_1, m_2, k_1, k_2$ satisfy the relations
\begin{equation}\label{eq13}
m_1=\frac{k_2(a_1b_2-a_2b_1)}{a_2(k_1-k_2)}, m_2=\frac{k_1(a_1b_2-a_2b_1)}{a_1(k_1-k_2)},\\
k_1\varphi(m_1)=k_2\varphi(m_2).
\end{equation}
and $q_i=k_il+1 (i=1, 2)$ are both prime.
If
\begin{equation}\label{eq132}
n=\frac{m_1q_1-b_1}{a_1},
\end{equation}
then we have $a_1n+b_1=m_1q_1, a_2n+b_2=m_2q_2$ and $\varphi(a_1n+b_1)=\varphi(a_2n+b_2)$.

Furthermore, if $a_1=a_2$ and the condition (\ref{eq13}) holds, then $m_1, m_2=m_1+b_2-b_1$
must have the same prime factors.
\end{thm}

\begin{thm}\label{th12}
Let $a_1, b_1, a_2, b_2$ be integers with $a_1>0, a_2>0, a_1b_2-a_2b_1\neq 0$.
Let $N(a_1, b_1, a_2, b_2; x)$ denote the number of integers $n\leq x$ with $\varphi(a_1n+b_1)=\varphi(a_2n+b_2)$
that are not in the form (\ref{eq122}) given in Theorem \ref{th11}.
Similarly, let $M(a_1, b_1, a_2, b_2; x)$ denote the number of integers $n\leq x$ with $\sigma(a_1n+b_1)=\sigma(a_2n+b_2)$
that are not in the form (\ref{eq132}) given in Theorem \ref{th11}.
Then $N(a_1, b_1, a_2, b_2; x)$ and $M(a_1, b_1, a_2, b_2; x)$ are both
$\ll x\exp(-(2^{-1/2}+o(1))(\log x \log\log\log x)^{1/2})$.
\end{thm}

Applied in some particular cases, these theorems give the following corollaries.

\begin{cor}\label{th13}
If $k$ is odd, then the number of solutions $n\leq x$ to $\varphi(n)=\varphi(n+k)$ is
$\ll x\exp(-(2^{-1/2}+o(1))(\log x \log\log\log x)^{1/2})$.
\end{cor}

\begin{cor}\label{th14}
If there exists no integer $m$ such that $\sigma(m)/m=\sigma(m+1)/(m+1)=k$ for some integer $k$,
then the number of solutions $n\leq x$ to $\sigma(n)=\sigma(n+1)$ is
$\ll x\exp(-(2^{-1/2}+o(1))(\log x \log\log\log x)^{1/2})$.
\end{cor}

The proof of Theorem \ref{th11} is straightforward.  The proof of Theorem \ref{th12}
depends on one of many results of Banks, Friedlander, Pomerance, Shparlinski \cite{BFPS}
concerning to multiplicative structures of values of Euler's totient function.

It is unlikely that there exists an integer $m$ such that $\sigma(m)/m=\sigma(m+1)/(m+1)=k$ for some integer $k$.
However, the proof of the nonexistence of such an integer will be difficult.
Luca \cite{Luc} shows that in the case $k=2$, such an integer never exists.
We note that the nonexistence of such an integer would follow from the conjecture that there
exists no odd integer $m>1$ for which $m$ divides $\sigma(m)$.

\section{Preliminary Lemmas}

In this section, we shall introduce some basic lemmas on distributions of integers
with special multiplicative structures.

We denote by $P(n), p(n)$ the largest and smallest prime factor of $n$ respectively.
We denote by $x, y, z$ real numbers and we put $u=\log x/\log y$
and $v=\log y/\log z$.  These notations are used in later sections.

\begin{lem}\label{lm21}
Denote by $\Psi(x, y)$ the number of integers $n\leq x$ divisible by no prime $>y$.
If $y>\log^2 x$, then we have $\Psi(x, y)<x\exp(-u\log u+o(u))$ as $x, u$ tend to infinity.
\end{lem}
\begin{proof}
This follows from a well-known theorem of de Bruijn \cite{Bru}.  A simpler proof
is given by Pomerance \cite{Pom}.
\end{proof}

\begin{lem}\label{lm22}
Let
\begin{equation}
S=\{n\mid p^2\mid n\mbox{ for some }p, a\mbox{ with }\sigma(p^a)>y, a\geq 2\}.
\end{equation}
Then we have the number of elements in $S$ below $x$ is $\ll xy^{-1/2}$.
\end{lem}
\begin{proof}
Let $\Pi(t)$ be the number of perfect powers below $t$.
It is clear that $\Pi(t)<t^{1/2}+t^{1/3}+\ldots+t^{1/k}<t^{1/2}+kt^{1/3}=t^{1/2}(1+o(1))$,
where $k=\floor{(\log t)/(\log 2)}$.

Let us denote by $\gamma_p$ the smallest integer $\gamma$ for which $\sigma(p^{\gamma})>y$ and $\gamma>1$.
Clearly we have $\#S(x)\leq x\sum_{p\leq x}p^{-\gamma_p}$.  Since $p^{\gamma_p}>\sigma(p^{\gamma_p})/2>y/2$,
we have by partial summation
\begin{equation}
\sum_{p\leq x}\frac{1}{p^{\gamma_p}}<\frac{\Pi(x)}{x}-\frac{\Pi(y/2)}{y}+\int_{y/2}^x\frac{\Pi(t)}{t^2}dt\ll y^{-1/2}.
\end{equation}
This proves the lemma.
\end{proof}

We use an upper bound for $\Phi(x, y)$ the number of integers
$n\leq x$ such that $P(\sigma(n))\leq y$ or $P(\varphi(n))\leq y$.

\begin{lem}\label{lm23}
Let $\Phi(x, y)$ denote the number of integers $n\leq x$ such that $P(\varphi(n))\leq y$.
and $\Sigma(x, y)$ denote the number of integers $n\leq x$ such that $P(\sigma(n))\leq y$.
For any fixed $\epsilon>0$ and $(\log\log x)^{1+\epsilon}< y\leq x$, we have
\begin{equation}
\Phi(x, y)\ll x\left(\exp(-u(1+o(1))\log\log u) \right),
\end{equation}
and
\begin{equation}
\Sigma(x, y)\ll x\max\{y^{-1/2}, \left(\exp(-u(1+o(1))\log\log u) \right)\},
\end{equation}
when $u=(\log x)/(\log y)\rightarrow\infty$.
\end{lem}
\begin{proof}
The first result is Theorem 3.1 in \cite{BFPS}.
The second result can be proved similarly, noting that the number of integers $n\leq x$
such that $p^a\mid n$ for some prime $p$ with $\sigma(p^a)>y, a\geq 2$ is
$\ll x/y^{-1/2}$ by Lemma \ref{lm22}.
\end{proof}

\section[Proof of Theorem 1.1]{Proof of Theorem \ref{th11}}

If (\ref{eq12}) holds and $q_i=k_il-1 (i=1, 2)$ are both prime,
then we have $\sigma(m_iq_i)=\sigma(m_i)k_il (i=1, 2)$ must be equal since
$k_1\sigma(m_1)=k_2\sigma(m_2)$.  Moreover, we have
$a_1m_2q_2-a_2m_1q_1=l(a_1m_2k_2-a_2m_1k_1)-(a_1m_2-a_2m_1)=a_1b_2-a_2b_1$
since $a_1m_2k_2=a_2m_1k_1$ and $a_1m_2-a_2m_1=-a_1b_2+a_2b_1$.

The corresponding statement with $\sigma$ replaced by $\varphi$ can be
similarly proved.

Finally, for the last statement, we can easily see that $m_1, m_2$ must have the same prime factors
since $\varphi(m_1)/m_1=\varphi(m_2)/m_2$.
This completes the proof of Theorem \ref{th11}.

\section[Proof of Theorem 1.2]{Proof of Theorem \ref{th12}}

We prove Theorem \ref{th12} for $\sigma$.  The corresponding statement for $\varphi$
can be proved in a similar, but slightly simpler way since we need not
to be careful about square factors.

Let $B(x)$ be the set of integers $n\leq x$ not in the form (\ref{eq122}) given in Theorem \ref{th11}
for which the equation $\sigma(a_1n+b_1)=\sigma(a_2n+b_2)$ hold.
All that we should prove is that $\#B(x)\ll x\exp(-(2^{-1/2}+o(1))(\log x \log\log\log x)^{1/2})$.
Of course, we may assume that $x$ is sufficiently large.

We put $y=\exp (2^{1/2} (\log x \log\log\log x)^{1/2}), z=y^{1/2}, z_1=z/\log x$
and $z_2=z\log x$.  Thus we have $u=(\log x)/(\log y)=2^{-1/2} (\log x)^{1/2} (\log\log\log x)^{-1/2}$.
Theorem \ref{th12} for $\sigma$ can be formulated that $\#B(x)<xz^{-1+o(1)}$.
We note that we can take $x$ to be sufficiently large so that $y$ is also sufficiently large.

Let us consider the following sets of integers:
\begin{equation*}
\begin{split}
B_1(x)=&\{n\mid n\in B(x), a_1n+b_1\in S\textrm{ or } a_2n+b_2\in S\},\\
B_2(x)=&\{n\mid n\in B(x), P(\sigma(a_1n+b_1))\leq y\},\\
B_0(x)=&B(x)\backslash (B_1(x)\cup B_2(x)).
\end{split}
\end{equation*}
We have $\#B_1(x)\ll xy^{-1/2}=x/z$ by Lemma \ref{lm22}.
Moreover, we have $\#B_2(x)\ll \max \{xy^{-1/2}, xz^{-1+o(1)}\}\ll xz^{-1+o(1)}$ by Lemma \ref{lm23}.

Now let $n\in B_0(x)$.  Since $n\not\in B_2(x)$, $\sigma(a_1n+b_1)$ must have
some prime factor $p>y$.  Therefore $a_1n+b_1$ must have some prime power
factor $q^a$ with $\sigma(q^a)>y$.  However, we must then have $a=1$ since
$n\not\in B_1(x)$.  So that $a_1n+b_1$ must have some prime factor of the form
$k_1p-1$, where $k_1\geq 1$ is an integer.  Similarly, $a_2n+b_2$ must have
some prime factor of the form $k_2p-1$, where $k_2\geq 1$ is an integer.
So that we can write
\begin{equation}\label{eq41}
a_in+b_i=m_i(k_ip-1) (i=1, 2),
\end{equation}
where $p$ is a prime greater than $y$ and $m_1, m_2, k_1, k_2$ are positive integers
such that $k_ip-1$ is a prime not dividing $m_i$ for each $i=1, 2$.  Now we have
\begin{equation}\label{eq42}
\sigma(m_1)k_1=\sigma(m_2)k_2
\end{equation}
since
\begin{equation}
\sigma(m_1)k_1p=\sigma(a_1n+b_1)=\sigma(a_2n+b_2)=\sigma(m_2)k_2p.
\end{equation}

Now we divide $B_0(x)$ into two sets
\begin{equation*}
B_3(x)=\{n\mid n\in B_0(x), a_in+b_i=m_i(k_ip-1), m_1m_2\leq x/z\}
\end{equation*}
and
\begin{equation*}
B_4(x)=\{n\mid n\in B_0(x), a_in+b_i=m_i(k_ip-1), m_1m_2>x/z\}.
\end{equation*}

We show that $\#B_3(x)<xz^{-1+o(1)}$.
Multiplying (\ref{eq41}) by $a_{3-i}$ and subtracting one from the other, we have
\begin{equation}\label{eq43}
a_2m_1(k_1p-1)-a_1m_2(k_2p-1)=a_2b_1-a_1b_2.
\end{equation}

Let us denote $c=\gcd (k_1, k_2)$ and $\overline{k_1}=k_1/c, \overline{k_2}=k_2/c$.
By virtue of (\ref{eq42}), $\overline{k_1}, \overline{k_2}$ are uniquely determined
by $m_1, m_2$.

If $a_2m_1k_1\neq a_1m_2k_2$, then $p$ can be expressed by
\begin{equation}
p=\frac{a_2(m_1+b_1)-a_1(m_2+b_2)}{a_2m_1k_1-a_1m_2k_2}.
\end{equation}
Therefore we have
\begin{equation}
pc=\frac{a_2(m_1+b_1)-a_1(m_2+b_2)}{a_2m_1\overline{k_1}-a_1m_2\overline{k_2}}.
\end{equation}
So that $c$ can be uniquely determined by $p, m_1, m_2$.
Since $p$ divides $d=a_2(m_1+b_1)-a_1(m_2+b_2)$, the number of possibilities of $p$
is at most $\omega(d)\ll\log d\ll\log (x/z)\ll\log x$.
Thus the number of possibilities of a pair $(m_1, m_2)$ is $\ll (x/z)\log x=x/z_1$.
Now we obtain $\#B_3(x)\ll x\log x/z_1=xz^{-1+o(1)}$ provided that $a_2m_1k_1\neq a_1m_2k_2$.

If $a_2m_1k_1=a_1m_2k_2$, then (\ref{eq43}) gives $a_2(b_1+m_1)=a_1(b_2+m_2)$.
Hence we have
\begin{equation}
m_1=\frac{k_2(a_1b_2-a_2b_1)}{a_2(k_2-k_1)}, m_2=\frac{k_1(a_1b_2-a_2b_1)}{a_1(k_2-k_1)}.
\end{equation}
Therefore, taking (\ref{eq42}) into account, $k_1, k_2, m_1, m_2, l=p, q_1=k_1p-1, q_2=k_2p-1$ satisfy the condition
in Theorem \ref{th11}.

Next we show that $\#B_4(x)$ is also at most $xz^{-1+o(1)}$.
If $k_1=k_2$, then $k_1p-1$ divides $a_1b_2-a_2b_1$.  Since we have taken $x$ sufficiently large,
we have $p-1>y-1>\abs{a_1b_2-a_2b_1}$ and therefore $a_1b_2-a_2b_1=0$, contrary to the assumption.
Hence we see that $k_1\neq k_2$.  Now, recalling that $q_1=k_1p-1$ and $q_2=k_2p-1$
are both prime, (\ref{eq41}) implies that
$n\equiv\psi(a_1, b_1, a_2, b_2, k_1, k_2)\pmod{(k_1p-1)(k_2p-1)}$,
where $\psi(a_1, b_1, a_2, b_2, k_1, k_2)$ is the unique simultaneous solution of
two congruences $a_1\psi(a_1, b_1, a_2, b_2, k_1, k_2)\equiv -b_1\pmod{k_1p-1}$
and $a_2\psi(a_1, b_1, a_2, b_2, k_1, k_2)\equiv -b_2\pmod{k_2p-1}$.

Since $m_1m_2\geq x/z$, we have $(k_1p-1)(k_2p-1)\leq (a_1x+b_1)(a_2x+b_2)/(x/z)\ll xz$.
Now the number of elements of $B_4(x)$ can be bounded by
\begin{equation}
\sum_{p, q_1, q_2}\left(\frac{x}{q_1q_2}+1\right),
\end{equation}
where $p$ runs over the primes greater than $y-1$ and $q_1, q_2$ run over the primes such that
$q_1\equiv q_2\equiv -1\pmod{p}$ and $q_1q_2\leq (a_1x+b_1)(a_2x+b_2)/(m_1m_2)\ll xz$.

For each prime $p$, the sum can be estimated as
\begin{equation}
\sum_{k_1, k_2}\left(\frac{x}{(k_1p-1)(k_2p-1)}+1\right)\ll x\left(\sum_{k}\frac{1}{kp}\right)^2+\frac{xz\log (xz)}{p^2}
\ll\frac{xz_2}{p^2},
\end{equation}
where $k$ runs all positive integers up to $cx$ for some suitable constant $c$.

Since $p>y-1$, we have
\begin{equation}
\#B_4(x)\ll\sum_{p>y-1}\frac{xz_2}{p^2}\ll\frac{xz_2}{y}\ll xz^{-1+o(1)}.
\end{equation}

Clearly $B(x)=\bigcup_{1\leq i\leq 4}B_i(x)$ and each $\#B_i(x)$ is at most
$xz^{-1+o(1)}$.  Therefore $\#B(x)<xz^{-1+o(1)}$.  This proves Theorem \ref{th12}
for $\sigma$.  As we noted in the beginning of this section, Theorem \ref{th12}
for $\varphi$ can be proved in a similar way.  Now the proof is complete.

\section{Proof of corollaries}

Assume $\varphi(n)=\varphi(n+k)$ with $n$ satisfying the condition of Theorem \ref{th11}
and let $m_1, m_2$ be as appear in Theorem \ref{th11}.
Then $m_1$ and $m_2=m_1+k$ must have the same prime factors.  Thus $k$ must be even.
This gives Corollary \ref{th13}.

We derive Corollary \ref{th14} from Theorem \ref{th12}.  Assume that $\sigma(n)=\sigma(n+1),
n=m_1q_1, n+1=m_2q_2$, $m_1, m_2, k_1, k_2$ satisfy the relations (\ref{eq12})
and $q_i=k_il-1 (i=1, 2)$ are both prime.  We can take $k_1, k_2$ to be relatively prime
by replacing $k_1, k_2, l$ by $k_1/\gcd (k_1, k_2), k_2/\gcd (k_1, k_2), l\gcd (k_1, k_2)$ respectively.
The relations (\ref{eq12}) give
\begin{equation}\label{eq51}
m_1=k_2/(k_2-k_1), m_2=k_1/(k_2-k_1)
\end{equation}
and
\begin{equation}\label{eq52}
k_1\sigma(m_1)=k_2\sigma(m_2).
\end{equation}
We show that $m_1=k_2, m_2=k_2$ and $k_2=k_1+1$.  If $d$ divides $k_2-k_1$,
then $d$ must divide $k_2$ and therefore $d$ must divide $(k_1, k_2)$.  So that
we must have $d=1$.  Therefore $k_2-k_1=1$.  Now (\ref{eq52}) gives
$k_1\sigma(k_1+1)=(k_1+1)\sigma(k_1)$.  Since $k_1, k_1+1$ are clearly relatively
prime, $k_1$ must divide $\sigma(k_1)$ and $k_1+1$ must divide $\sigma(k_1+1)$.
Now we have $\sigma(k_1)=kk_1$ and $\sigma(k_1+1)=k(k_1+1)$ for some integer $k$.
This proves Corollary \ref{th14}.

{}
\vskip 8pt
{\small Center for Japanese language and culture, Osaka University,\\ 562-8558, 8-1-1, Aomatanihigashi, Minoo, Osaka, Japan}\\
{\small e-mail: \protect\normalfont\ttfamily{tyamada1093@gmail.com} URL: \url{http://tyamada1093.web.fc2.com/math/}
\end{document}